\documentclass[reqno,a4paper,11pt]{amsart} 

\usepackage[utf8]{inputenc}
\usepackage[english]{babel}
\usepackage{fullpage}
\usepackage{tikz}
\usepackage{amsmath}
\usepackage{amssymb}
\usepackage{amsthm}
\usepackage{thmtools}
\usepackage[colorlinks = true, linkbordercolor = {white}, urlcolor={purple}, linkcolor={purple}, citecolor = {purple}]{hyperref}
\usepackage{cleveref}
\usepackage{amsfonts}
\usepackage{enumerate}
\usepackage[foot]{amsaddr}
\usepackage{mathtools}
\usepackage[normalem]{ulem}
\usepackage{caption}
\usepackage{bbm}

\declaretheorem[name=Theorem]{THM:positivity-char}
\theoremstyle{definition}
\newtheorem{theorem}{Theorem}[section]
\newtheorem{definition}[theorem]{Definition}
\newtheorem{corollary}[theorem]{Corollary}
\newtheorem{prop}[theorem]{Proposition}
\newtheorem{lemma}[theorem]{Lemma}
\newtheorem{example}[theorem]{Example}

\theoremstyle{remark}

\newcommand{\Z}{{\mathbb Z}}

\newcommand{\N}{{\mathbb N}}

\newcommand{\mc}{\mathcal}

\title{On word complexity and topological entropy of random substitution subshifts}
\author{Andrew Mitchell}
\address{School of Mathematics, University of Birmingham, Edgbaston, B15 2TT, UK.}



\begin{document}

\begin{abstract}
    We consider word complexity and topological entropy for random substitution subshifts.
    In contrast to previous work, we do not assume that the underlying random substitution is compatible. 
    We show that the subshift of a primitive random substitution has zero topological entropy if and only if it can be obtained as the subshift of a deterministic substitution, answering in the affirmative an open question of Rust and Spindeler. 
    For constant length primitive random substitutions, we develop a systematic approach to calculating the topological entropy of the associated subshift.
    Further, we prove lower and upper bounds that hold even without primitivity.
    For subshifts of non-primitive random substitutions, we show that the complexity function can exhibit features not possible in the deterministic or primitive random setting, such as intermediate growth, and provide a partial classification of the permissible complexity functions for subshifts of constant length random substitutions.
\end{abstract}

\maketitle

\section{Introduction}

Topological entropy is a well-studied invariant of dynamical systems that quantifies the complexity of a given system.
In symbolic dynamics, the main objects of study are \emph{subshifts}.
For a given subshift $X$, the \emph{complexity function} $p_X \colon \N \rightarrow \N$ is the function that, for each $n \in \N$, returns the number of distinct words of length $n$ in the language of the subshift $X$.
The topological entropy of $X$ quantifies the exponential growth rate of the complexity function. Specifically, the topological entropy of $X$ is the quantity defined by
\begin{equation*}
    h_{\operatorname{top}} (X) = \lim_{n \rightarrow \infty} \frac{1}{n} \log p_X (n) \text{;}
\end{equation*}
that the limit exists follows by the sub-additivity of $\log p_X (n)$ and Fekete's lemma \cite{lind-marcus}. 

Subshifts associated with deterministic substitutions are the prototypical examples of mathematical quasicrystals and are well-studied dynamical systems with zero topological entropy \cite{queffelec}.
It was shown by Pansiot \cite{pansiot} that the complexity function of a deterministic substitution subshift is at most quadratic, and falls into one of five complexity classes depending on properties of the substitution.
Further, for \emph{primitive} substitutions, the complexity function either grows linearly or is bounded above by a constant.
Random substitutions are a generalisation of deterministic substitutions where, instead of replacing letters deterministically, the substituted image of a letter is chosen from a fixed finite set according to a probability distribution. 
Similarly to deterministic substitutions, a subshift can be associated with a given random substitution in a canonical manner.
However, in contrast to their deterministic counterparts, random substitution subshifts often have positive topological entropy \cite{gohlke,rust-spindeler}.
Random substitutions have gained increased attention in the mathematics community in recent years, as they give rise to subshifts with many interesting properties.
For instance, they provide examples of positive entropy subshifts that do not admit periodic points \cite{rust-periodic-points} and intrinsically ergodic subshifts that do not satisfy the classical specification property \cite{MT-entropy}.
Moreover, a broad variety of well-known subshifts can be obtained as subshifts of random substitutions, including all topologically transitive shifts of finite type \cite{gohlke-rust-spindeler} and all deterministic substitution subshifts \cite{rust-spindeler}.

The non-trivial topological entropy of random substitution subshifts provides a new invariant in their study not available for deterministic substitutions.
While random substitution subshifts typically have positive topological entropy, every deterministic substitution subshift is itself the subshift of a random substitution, so it is not true that all random substitution subshifts have positive topological entropy.
Rust and Spindeler \cite{rust-spindeler} conjectured that for \emph{primitive} random substitutions, the zero entropy subshifts are precisely those that can be obtained as subshifts of deterministic substitutions.
In \Cref{THM:positivity-char}, we provide a proof of this conjecture. 
As a consequence of \Cref{THM:positivity-char}, combined with the classification of complexity for deterministic substitutions given by Pansiot in \cite{pansiot}, we deduce that the complexity function of a primitive random substitution subshift grows exponentially or linearly, or is bounded above by a constant.

A systematic approach to calculating the topological entropy of random substitution subshifts was recently provided by Gohlke \cite{gohlke}.
There, it was shown that for \emph{compatible} primitive random substitutions, the topological entropy of the associated subshift coincides with the notion of \emph{inflation word entropy}, which is defined in terms of the underlying random substitution (as opposed to the subshift).
This provides a means of accurately estimating topological entropy and, in many cases, yields an exact formula.
However, a limitation is the assumption of compatibility.
The random substitutions that give rise to many well-known subshifts, such as shifts of finite type, are typically not compatible.
As such, it is desirable to better understand topological entropy for non-compatible random substitution subshifts.
The main obstacle to extending the approach in \cite{gohlke} to non-compatible random substitutions is that their subshifts do not have uniform letter frequencies, a property fundamental to the proofs in \cite{gohlke}.
In \Cref{S:CL-TE}, we develop techniques to circumvent this issue for \emph{constant length} random substitutions. 
We show that inflation word entropy and topological entropy coincide for all subshifts of constant length primitive random substitutions (\Cref{THM:inf-entropy}), providing a means of calculating the topological entropy for a broad class of non-compatible random substitutions.
Further, even without primitivity, we obtain lower and upper bounds for the topological entropy in terms of a quantity defined similarly to inflation word entropy.

While for primitive random substitutions the associated subshift has zero topological entropy if and only if it is the subshift of a deterministic substitution, the same does not hold without primitivity.
In fact, non-primitive random substitutions give rise to a rich variety of zero topological entropy subshifts with properties not witnessed for deterministic substitution subshifts.
In \Cref{S:CL-complexity}, we present several results on complexity for subshifts of non-compatible random substitutions of constant length.
We show that such random substitutions can give rise to subshifts with intermediate growth complexity function and provide sufficient conditions for this phenomena to occur.
Moreover, we show that there is no polynomial complexity gap for such subshifts.
This is in stark contrast to subshifts of deterministic substitutions, where the asymptotic growth rate of the complexity function is never faster than quadratic.

\section{Random substitution subshifts}\label{S:r-subst-prelim}

\subsection{Notation}

Throughout, we use the following symbolic notation. 
An \emph{alphabet} $\mc{A}$ is a finite collection of symbols, which we call \emph{letters}. 
We call a finite concatenation of letters a \emph{word}, and let $\mc A^{+}$ denote the set of all non-empty finite words with letters from $\mc A$. 
We write $\lvert u \rvert$ for the length of $u$ and, for each $a \in \mc A$, let $\lvert u \rvert_a$ denote the number of occurrences of $a$ in $u$. 
We let $\mc{A}^{\Z}$ denote the set of all bi-infinite sequences of elements in $\mc{A}$ and endow $\mc{A}^{\Z}$ with the discrete product topology. 
With this topology, the space $\mc{A}^{\Z}$ is compact and metrisable. 
We let $S$ denote the usual (left-)shift map. 
If $i, j \in \mathbb{Z}$ with $i \leq j$, and $x = \cdots x_{-1} x_{0} x_{1} \cdots \in \mc{A}^{\mathbb{Z}}$, then we write $x_{[i,j]} = x_i x_{i+1} \cdots x_{j}$. 
A \emph{subshift} $X$ is a closed and $S$-invariant subspace of $\mc A^{\Z}$. 
We let $\mc L (X)$ denote the \emph{language} of the subshift $X$: namely, the set of all finite words $u \in \mc A^{+}$ such that there exist $x \in X$ and $j \in Z$ for which $x_{[j,j+\lvert u \rvert-1]} = u$.
Moreover, for each $n \in \N$, we write $\mc L^n (X) = \{ u \in \mc L (X) \colon \lvert u \rvert = n \}$. 

For a given set $B$, we let $\# B$ be the cardinality of $B$ and let $\mc{F}(B)$ be the set of non-empty finite subsets of $B$.
Given functions $f, g \colon \N \rightarrow \N$, we write $g = \Theta (f)$ if there exist constants $C_1, C_2 > 0$ such that $C_1 < f(n)/g(n) < C_2$ for all $n \in \N$.

\subsection{Random substitutions and their subshifts}

We now give the definition of our main objects of study, random substitutions.

\begin{definition}
Let $\mc{A} = \{ a_{1}, \ldots, a_{d} \}$ be a finite alphabet.
A random substitution $(\vartheta, \mathbf{P})$ is a finite-set-valued function $\vartheta \colon \mathcal{A} \rightarrow \mathcal{F}(\mathcal{A}^{+})$ together with a set of non-degenerate probability vectors 
	\begin{align*}
	\mathbf{P} = \left\{ \mathbf{p}_i = ( p_{i, 1}, \ldots, p_{i, k_i} ) : k_i = \# \vartheta(a_i), \, \mathbf{p}_i \in (0,1]^{k_i} \text{ and } \sum_{j=1}^{k_i} p_{i,j} = 1 \text{ for all } 1 \leqslant i \leqslant d \right\},
	\end{align*}
such that
	\begin{align*}
	\vartheta \colon a_i \mapsto
		\begin{cases}
		s^{(i,1)} & \text{with probability } p_{i, 1},\\
		\hfill \vdots \hfill & \hfill \vdots\hfill\\
		s^{(i,k_i)} & \text{with probability } p_{i, k_i},
		\end{cases}
	\end{align*}
for every $1 \leqslant i \leqslant d$, where $\vartheta(a_i) = \{ s^{(i,j)} \}_{1\leqslant j \leqslant k_i}$.
For each $a \in \mc A$, we call a word $u \in \vartheta (a)$ a \emph{realisation} of $\vartheta (a)$.
If there exists an integer $\ell \geq 2$ such that for all $a \in \mc A$ and $v \in \vartheta (a)$ we have $\lvert v \rvert = \ell$, then we call $\vartheta$ \emph{constant length}.
If $\# \vartheta (a) = 1$ for all $a \in \mc A$, then we call $\vartheta$ \emph{deterministic}. 
We say that a deterministic substitution $\theta$ is a \emph{marginal} of $\vartheta$ if, for all $i \in \{1,\ldots,d\}$, we have $\theta \colon a_i \mapsto s^{(i,j)}$ for some $j \in \{ 1,\ldots,k_i\}$.
\end{definition}

The subshift associated with a random substitution (see \Cref{DEF:r-subst-subshift}) is independent of the choice of non-degenerate probabilities prescribed to the random substitution.
As such, we omit the explicit dependence on the probabilities in our notation throughout.

\begin{example}\label{EX:first-examples}
    Let $\vartheta_1$, $\vartheta_2$, $\vartheta_3$ and $\vartheta_4$ be the random substitutions defined by
    \begin{equation*}
        \begin{split}
            \vartheta_1 \colon
            \begin{cases}
                a \mapsto 
                \begin{cases}
                    ab\text{,}\\
                    ba\text{,}
                \end{cases}\\
                b \mapsto a\text{,}
            \end{cases}
            \quad \vartheta_2 \colon
            \begin{cases}
                a \mapsto 
                \begin{cases}
                    aa\text{,}\\
                    ab\text{,}
                \end{cases}\\
                b \mapsto ba\text{,}
            \end{cases}
            \quad \vartheta_3 \colon
            \begin{cases}
                a \mapsto aaa\text{,}\\
                b \mapsto
                \begin{cases}
                    bba\text{,}\\
                    abb\text{,}
                \end{cases}
            \end{cases}
            \quad \vartheta_4 \colon
            \begin{cases}
                a \mapsto abc\text{,}\\
                b \mapsto bca\text{,}\\
                c \mapsto cab\text{.}
            \end{cases}
        \end{split}
    \end{equation*}
    The random substitutions $\vartheta_1$, $\vartheta_2$ and $\vartheta_3$ are defined over the two-letter alphabet $\{a,b\}$ and $\vartheta_4$ is defined over the three-letter alphabet $\{a,b,c\}$. 
    The random substitutions $\vartheta_2$, $\vartheta_3$ and $\vartheta_4$ are constant length and $\vartheta_4$ is deterministic; $\vartheta_1$ is the well-studied \emph{random Fibonacci substitution}.
\end{example}

The action of a random substitution extends to finite words in a natural way, by applying it independently to each letter and concatenating in the order prescribed by the initial word---see \cite{MT-entropy} for more details.
For a given random substitution $\vartheta$ and $u \in \mc A^{+}$, we let $\vartheta (u)$ denote the set of all realisations of $u$ under $\vartheta$.
Higher powers of $\vartheta$ can be defined inductively as follows: for each $k \in \N$ and $a \in \mc A$, we let $\vartheta^k (a) = \cup_{v \in \vartheta^{k-1} (a)} \vartheta (v)$.

To a given random substitution, a subshift can be associated in a natural way.

\begin{definition}
    Given a random substitution $\vartheta$, we say that a word $u \in \mc{A}^{+}$ is \emph{($\vartheta$-)legal} if there exist an $a \in \mc{A}$ and $k \in \mathbb{N}$ such that $u$ appears as a subword of some word in $\vartheta^{k} (a)$. 
    We define the \emph{language of $\vartheta$} by $\mc{L}_{\vartheta} = \{ u \in \mc{A}^{+} : u \text{ is $\vartheta$-legal} \}$. 
    Similarly, for $w \in \mc{A}^{+} \cup \mc{A}^{\mathbb{Z}}$, we define the \emph{language of $w$} by $\mc{L} (w) = \{ u \in \mc{A}^{+} : u \text{ is a subword of } w \}$.
\end{definition}

We call $u \in \mc A^{+}$ an \emph{inflation word} if there exists a $v \in \mc L_{\vartheta}$ such that $u \in \vartheta (v)$.
Moreover, we call $u$ an \emph{exact inflation word} if there exists an $a \in \mc A$ such that $u \in \vartheta (a)$.

\begin{definition}\label{DEF:r-subst-subshift}
The \emph{subshift} associated with a random substitution $\vartheta$ is the system $(X_{\vartheta}, S)$, where $X_{\vartheta} = \{ w \in \mc{A}^{\mathbb{Z}} : \mc{L} (w) \subseteq \mc{L}_{\vartheta} \}$ and $S$ is the (left) shift map.
\end{definition}

We endow $X_{\vartheta}$ with the subspace topology inherited from $\mc{A}^{\mathbb{Z}}$. Since $X_{\vartheta}$ is defined in terms of a language, it is a compact $S$-invariant subspace of $\mc{A}^{\mathbb{Z}}$; hence, $X_{\vartheta}$ defines a subshift. 
For $n \in \mathbb{N}$, we write $\mc{L}_{\vartheta}^{n} = \mc{L}_\vartheta \cap \mc{A}^{n}$ for the subset of $\mc{L}_{\vartheta}$ consisting of words of length $n$.  

The set-valued function $\vartheta$ extends naturally to $X_{\vartheta}$, where for $w = \cdots w_{-1} w_{0} w_{1} \cdots \in X_{\vartheta}$ we let $\vartheta(w)$ denote the (possibly uncountable) set of sequences of the form $v = \cdots v_{-2} v_{-1}.v_0 v_1 \cdots$, where $v_j \in \vartheta(w_j)$ for all $j \in \mathbb{Z}$. 
We write $\vartheta (X_{\vartheta})$ for the set of bi-infinite sequences $x \in \mc A^{\Z}$ for which there exists a $y \in X_{\vartheta}$ such that $x \in \vartheta (y)$.
By definition, we have $\vartheta(X_{\vartheta}) \subseteq X_{\vartheta}$.

\subsection{Primitive and compatible random substitutions}\label{SS:prim-comp}

Random substitutions give rise to a broad variety of subshifts, so it is natural to impose additional assumptions in their study.
Two of the most commonly imposed assumptions are \emph{primitivity} and \emph{compatibility}.
These conditions are assumed in, for example, \cite{baake-spindeler-strungaru,gohlke,miro-et-al,mitchell-rutar,rust-periodic-points}.

\begin{definition}
    Let $\vartheta$ be a random substitution. We say that $\vartheta$ is \emph{primitive} if there exists a positive integer $k$ such that for every pair $a,b \in \mc A$, the letter $a$ appears as a subword of at least one realisation of $\vartheta^k (b)$.
\end{definition}

While primitive random substitutions can give rise to empty subshifts, this only happens when, for all $a \in \mc A$, $\vartheta (a)$ consists only of realisations of length $1$ \cite[Prop. 9]{rust-spindeler}.
For simplicity, we exclude these pathological cases and from now on restrict the definition of primitivity to those random substitutions that give rise to a non-empty subshift.

It follows routinely from the definition that if a random substitution $\vartheta$ is primitive, then $\mc L (X_{\vartheta}) = \mc L_{\vartheta}$ and, for all $k \in \N$, the random substitution $\vartheta^k$ gives rise to the same subshift $X_{\vartheta}$. 

An important consequence of primitivity is the existence of elements in the associated subshift for which every letter occurs with positive frequency.
Namely, if $\vartheta$ is primitive, there exist $x \in X_{\vartheta}$ for which $\lim_{n \rightarrow \infty} \lvert x_{[-n,n]} \rvert_a / (2n+1)$ exists and is strictly positive for all $a \in \mc A$.
In fact, this is true for almost every element of the subshift with respect to a suitably chosen ergodic probability measure \cite{gohlke-spindeler}.
However, in general, $x \in X_{\vartheta}$ need not even have well-defined letter frequencies. 
This possibility can be excluded if $\vartheta$ is additionally assumed to be compatible.

\begin{definition}
    We say that a random substitution $\vartheta$ is \emph{compatible} if for every $a \in \mc A$ and $u,v \in \vartheta(a)$, we have $\lvert u \rvert_b = \lvert v \rvert_b$ for all $b \in \mc A$.
\end{definition}

If $\vartheta$ is compatible, then for every $a \in \mc A$, all realisations of $\vartheta (a)$ have the same length. 
We denote this common length by $\lvert \vartheta (a) \rvert$. 
Both primitivity and compatibility are preserved under taking powers: namely, if $\vartheta$ is primitive (or compatible), then for all $k \in \N$ the random substitution $\vartheta^k$ is primitive (or compatible). 

Together, primitivity and compatibility guarantee the existence of uniformly well-defined letter frequencies and a well-defined inflation rate.

\begin{lemma}[{\cite[Prop.~14]{gohlke}}]\label{LEM:PF-bounds}
    Let $\vartheta$ be a compatible primitive random substitution.
    Then, there exists a real number $\lambda > 1$ and non-degenerate probability vector $\mathbf{R} \in (0,1]^{\# \mc A}$ such that for every $\varepsilon > 0$, there exists an $N \in \N$ for which every legal word $v$ of length at least $N$ satisfies $(\lambda-\varepsilon) \lvert v \rvert \leq \lvert \vartheta (v) \rvert \leq (\lambda+\varepsilon) \lvert v \rvert$ and $(R_a - \varepsilon) \lvert v \rvert \leq \lvert v \rvert_a \leq (R_a + \varepsilon) \lvert v \rvert$ for all $a \in \mc A$.
    Moreover, for every $x \in X_{\vartheta}$ and $a \in \mc A$, we have $\lvert x_{[-n,n]} \rvert_a / (2n+1) \rightarrow R_a$ as $n \rightarrow \infty$.
\end{lemma}

The following condition is related to compatibility, and encompasses all compatible and all constant length random substitutions.

\begin{definition}\label{DEF:URP}
We say that $\vartheta$ has \emph{unique realisation paths} if for every $v \in \mc L_{\vartheta}$ and $k \in \N$ the concatenation map 
\begin{equation*}
    \vartheta^k(v_1) \times \cdots \times \vartheta^k(v_{\lvert v \rvert}) \to \mc L_{\vartheta}, \quad
    (w_1,\ldots,w_{\lvert v \rvert}) \mapsto w_1 \cdots w_{\lvert v \rvert}
\end{equation*}
is injective.
\end{definition}

\begin{lemma}[{\cite[Lemma~2.11]{MT-entropy}}]
    Let $\vartheta$ be a primitive random substitution. If $\vartheta$ is compatible or constant length, then $\vartheta$ has unique realisation paths.
\end{lemma}

Of the random substitutions defined in \Cref{EX:first-examples}, $\vartheta_1$, $\vartheta_2$ and $\vartheta_3$ are primitive. The random substitutions $\vartheta_1$, $\vartheta_3$ and $\vartheta_4$ are compatible, and all four have unique realisation paths as they are all compatible or constant length.

\section{Positivity of topological entropy}\label{S:positivity}

Positivity of topological entropy for random substitution subshifts was first identified in the pioneering work of Godr\`{e}che and Luck \cite{godreche-luck}. 
Since then, it has been shown that a wide range of random substitutions give rise to positive entropy subshifts (see, for example \cite{gohlke,koslicki,nilsson,spindeler-thesis,wing}).
Of course, every deterministic substitution subshift is itself a random substitution subshift, so not all subshifts of random substitutions have positive topological entropy.
Rust and Spindeler \cite{rust-spindeler} conjectured that, under primitivity, the zero entropy random substitution subshifts are precisely those that can be obtained as the subshift of a deterministic substitution.
Here, we provide a proof of this conjecture.

\begin{theorem}\label{THM:positivity-char}
    Let $X$ be the subshift of a primitive random substitution. Then $h_{\operatorname{top}} (X) = 0$ if and only if $X$ is the subshift of a (primitive) deterministic substitution.
\end{theorem}

By combining \Cref{THM:positivity-char} with the classification of complexity functions for subshifts of primitive deterministic substitutions given by Pansiot \cite{pansiot}, we obtain the following.
Consequently, we deduce that there do not exist primitive random substitution subshifts with intermediate growth complexity function.

\begin{corollary}\label{COR:no-int-growth}
    If $\vartheta$ is a primitive random substitution such that $h_{\operatorname{top}} (X_{\vartheta}) = 0$, then the complexity function of $X_{\vartheta}$ is either $\Theta(1)$ or $\Theta(n)$.
\end{corollary}

\Cref{COR:no-int-growth} illustrates that there exists a \emph{complexity gap} for subshifts of primitive random substitutions.
In particular, any function that grows faster than linearly but sub-exponentially cannot be obtained as the complexity function of a primitive random substitution subshift.
On the other hand, there is no exponential complexity gap, as it was shown in \cite[Theorem~42]{gohlke-rust-spindeler} that the set $\{ h_{\operatorname{top}} (X_{\vartheta}) \colon \vartheta \text{ is a primitive random substitution} \}$ is dense in $[0,\infty)$.

Without primitivity, the conclusion of \Cref{THM:positivity-char} does not hold.
In \Cref{S:CL-complexity}, we show that non-primitive random substitutions give rise to subshifts whose complexity functions exhibit a rich variety of behaviour not witnessed for subshifts of deterministic or primitive random substitutions.
In particular, we show that there exist zero entropy subshifts of non-primitive random substitutions that have intermediate growth complexity function, as well as ones with polynomial growth not possible for deterministic substitution subshifts.
A key feature of these examples is the existence of letters that occur with frequency zero in every element of the subshift, which is not possible under primitivity.

We now turn towards the proof of \Cref{THM:positivity-char}. 
Our proof uses the notion of a \emph{splitting pair} for a random substitution, which was introduced by Rust and Spindeler \cite{rust-spindeler}.
There, they showed that this notion provides sufficient conditions for positive entropy of the associated subshift.

\begin{definition}
Given $u,v \in \mc{A}^{+}$, we say that $u$ is an \emph{affix} of $v$ if $u$ is either a prefix or a suffix of $v$, namely, if $u = v_{[1,\lvert u \rvert]}$ or $u = v_{[\lvert v \rvert - \lvert u \rvert + 1, \lvert v \rvert]}$.
If $u$ is both a prefix and a suffix, then we call $u$ a \emph{strong affix} of $v$. 
\end{definition}

\begin{definition}
If $\vartheta$ is a random substitution and $a \in \mc{A}$ is such that there exist realisations $u,v \in \vartheta (a)$ with $\lvert u \rvert \leq \lvert v \rvert$ for which $u$ is not a strong affix of $v$, then we say that $a$ \emph{admits a splitting pair for $\vartheta$}.
\end{definition}

\begin{prop}[{\cite[Corollary~34]{rust-spindeler}}]\label{PROP:splitting-pair}
    If $\vartheta$ is a primitive random substitution for which there exists a letter that admits a splitting pair, then $h_{\operatorname{top}} (X_{\vartheta}) > 0$.
\end{prop}

If $\vartheta$ is a primitive random substitution, then for all $m \in \N$ the random substitution $\vartheta^m$ gives rise to the same subshift as $\vartheta$.
Hence, to ascertain that the subshift $X_{\vartheta}$ has positive topological entropy, it is sufficient to verify that there exists a positive integer $m$ such that $\vartheta^m$ admits a splitting pair. Thus, to obtain \Cref{THM:positivity-char}, it suffices to show that every primitive random substitution for which no power admits a splitting pair can be obtained as the subshift of a deterministic substitution.
In particular, we show that any such random substitution gives rise to the same subshift as one of its marginals. 

\begin{proof}[Proof of \Cref{THM:positivity-char}]
    Let $\vartheta$ be a primitive random substitution with $h_{\operatorname{top}} (X_{\vartheta}) = 0$. 
    It follows by \Cref{PROP:splitting-pair} that no letter in $\mc A$ admits a splitting pair for any power of $\vartheta$.
    Thus, for every $m \in \N$ and $a \in \mc A$, there exists a unique realisation of $\vartheta^m (a)$ of greatest length, since if $u,v \in \vartheta^m (a)$ have $\lvert u \rvert = \lvert v \rvert$ and $u$ is a strong affix of $v$, then we require $u = v$.
    Let $\theta$ be the marginal of $\vartheta$ that maps each letter $a \in \mc A$ to the realisation of $\vartheta (a)$ of greatest length. 
    By construction, for each $a \in \mc A$ and $m \in \N$, $\theta^m (a)$ is the realisation of $\vartheta^m (a)$ of greatest length. 
    Since, for every $a \in \mc A$, $a$ does not admit a splitting pair for $\vartheta$, every realisation of $\vartheta (a)$ appears as a subword of $\theta (a)$, so primitivity transfers to $\theta$.
    Now, let $u \in \mc L (X_{\vartheta})$. 
    Since $\vartheta$ is primitive, we have $\mc L (X_{\vartheta}) = \mc L_{\vartheta}$, so there exist $m \in \N$, $a \in \mc A$ and $v \in \vartheta^m (a)$ such that $u$ appears as a subword of $v$.
    But every realisation of $\vartheta^m (a)$ appears as a subword of $\theta^m (a)$, so $u \in \mc L_{\theta}$. Primitivity of $\theta$ gives that $\mc L (X_{\theta}) = \mc L_{\theta}$, so we have that $u \in \mc L (X_{\vartheta})$. 
    Hence, $\mc L (X_{\vartheta}) = \mc L (X_{\theta})$ and so we conclude that $X_{\vartheta} = X_{\theta}$.
\end{proof}

While \Cref{THM:positivity-char} guarantees that every zero entropy primitive random substitution subshift can be obtained as the subshift of a deterministic substitution, it does not guarantee that \emph{every} random substitution that gives rise to that subshift must be deterministic. 
For example, the primitive random substitution $\vartheta \colon a \mapsto \{a,aba\}, \, b \mapsto \{bab\}$ gives rise to the finite subshift $X_{\vartheta} = \{(ab)^{\Z}, (ba)^{\Z} \}$. 
However, if a primitive random substitution is additionally assumed to have unique realisation paths (recall \Cref{DEF:URP}), then its subshift has zero topological entropy if and only if it is itself deterministic. 

\begin{prop}\label{PROP:URP-positivity}
    If $\vartheta$ is a primitive random substitution with unique realisation paths and there exists a letter $b \in \mc A$ such that $\# \vartheta (b) \geq 2$, then $h_{\operatorname{top}} (X_{\vartheta}) > 0$.
\end{prop}

\Cref{PROP:URP-positivity} is a consequence of a lower bound on measure theoretic entropy proved in \cite{MT-entropy}.
In particular, it follows by \cite[Theorem~3.5]{MT-entropy} that for every random substitution $\vartheta$ satisfying the conditions of \Cref{PROP:URP-positivity}, there exists an ergodic probability measure $\mu$, supported on the subshift $X_{\vartheta}$, such that $X_{\vartheta}$ has positive (measure theoretic) entropy with respect to the measure $\mu$.
Positivity of topological entropy then follows from the fact that measure theoretic entropy provides a lower bound for topological entropy.

\section{Quantitative results on topological entropy}\label{S:CL-TE}

While \Cref{THM:positivity-char} provides a classification of positive topological entropy for primitive random substitution subshifts, it does not provide a means of calculating or accurately estimating the topological entropy for the subshift of a given random substitution.
Since primitive random substitutions give rise to a broad variety of subshifts, it is reasonable to expect that extra assumptions will be required to obtain precise quantitative results.
Gohlke \cite{gohlke} showed that for all compatible primitive random substitutions, the topological entropy of the associated subshift coincides with the notion of \emph{inflation word entropy}.
In this section, we show that the same holds for all constant length primitive random substitutions.
This allows the topological entropy to be calculated for a broad class of non-compatible random substitution subshifts, which so far has only been possible for isolated examples.

\subsection{Inflation word entropy}

Let $\vartheta$ be a random substitution such that, for all $m \in \N$ and $a \in \mc A$, every realisation of $\vartheta^m (a)$ has the same length.
This holds, for example, if $\vartheta$ is constant length or compatible.
Letting $\lvert \vartheta^m (a) \rvert$ denote the common length of realisations of $\vartheta^m (a)$, the lower and upper inflation word entropy of type $a$, respectively, are defined by
\begin{equation*}
    \begin{split}
        \underline{h}_a &= \liminf_{m \rightarrow \infty} \frac{1}{\lvert \vartheta^m (a) \rvert} \log (\# \vartheta^m (a)) \text{,} \\
        \overline{h}_a &= \limsup_{m \rightarrow \infty} \frac{1}{\lvert \vartheta^m (a) \rvert} \log (\# \vartheta^m (a)) \text{.}
    \end{split}
\end{equation*}
When these limits coincide, we denote their common value by $h_a$, which we call the \emph{inflation word entropy} of type $a$. 

\begin{theorem}\label{THM:inf-entropy}
    Let $\vartheta$ be a primitive random substitution of constant length. Then, for all $a \in \mc A$, the inflation word entropy $h_a$ exists and coincides with $h_{\operatorname{top}} (X_{\vartheta})$.
\end{theorem}

We present the proof of \Cref{THM:inf-entropy} in \Cref{SS:inf-entropy-proof}.
In general, the conclusion of \Cref{THM:inf-entropy} does not hold without primitivity. 
For example, consider the non-primitive random substitution defined by $\vartheta \colon a \mapsto \{ab,ba\}, \, b \mapsto \{aa\}, \, c \mapsto \{cc\}$.
The associated subshift $X_{\vartheta}$ contains the subshift of the random period doubling substitution $\varphi \colon a \mapsto \{ab,ba\}, \, b \mapsto \{aa\}$ as a subsystem, which is known to have positive topological entropy \cite{gohlke,spindeler-thesis}; hence, $h_{\operatorname{top}} (X_{\vartheta}) \geq h_{\operatorname{top}} (X_{\varphi}) > 0$.
However, since $\# \vartheta^m (c) = 1$ for all $m \in \N$, we have $h_c = 0$.
In fact, one can show that $h_{\operatorname{top}} (X_{\vartheta}) = h_{\operatorname{top}} (X_{\varphi}) = \log(4)/3$.

\Cref{THM:inf-entropy} allows a closed-form formula for the topological entropy to be obtained for many constant length random substitution subshifts.
We now present two examples of non-compatible random substitutions where this is possible.

\begin{example}\label{EX:non-compatible-first-example}
    Let $\vartheta$ be the primitive random substitution defined by $\vartheta \colon a \mapsto \{aa,bb\}, \, b \mapsto \{aa\}$,
    and let $X_{\vartheta}$ denote the corresponding subshift. 
    Using \Cref{THM:inf-entropy}, we show that
    \begin{equation*}
        h_{\operatorname{top}} (X_{\vartheta}) = \frac{1}{2} \log 2 \text{.}
    \end{equation*}
    Observe that $\vartheta^m (b) \subseteq \vartheta^m (a)$ for all $m \in \N$.
    Hence, by the constant length property, we have
    \begin{equation*}
        \vartheta^m (a) = \vartheta^{m-1} (a) \vartheta^{m-1} (a) \cup \vartheta^{m-1} (b) \vartheta^{m-1} (b) = \vartheta^{m-1} (a) \vartheta^{m-1} (a)
    \end{equation*}
    for all $m \in \N$, so $\# \vartheta^m (a) = (\# \vartheta^{m-1} (a))^2$.
    It follows inductively that
    \begin{equation*}
        \frac{1}{2^m} \log (\# \vartheta^{m} (a)) = \frac{1}{2} \log (\# \vartheta (a)) = \frac{1}{2} \log 2
    \end{equation*}
    for all $m$, so we conclude by \Cref{THM:inf-entropy} that $h_{\operatorname{top}} (X_{\vartheta}) = h_a = \log(2)/2$.
\end{example}

\begin{example}\label{EX:non-compatible-second-example}
Let $\vartheta$ be the primitive random substitution defined by $\vartheta \colon a \mapsto \{aa,ab\}, \, b \mapsto \{ba\}$,
and let $X_{\vartheta}$ denote the corresponding subshift. The topological entropy of $X_{\vartheta}$ is
\begin{equation*}
    h_{\operatorname{top}} (X_{\vartheta}) = \sum_{n=1}^{\infty} \frac{1}{2^n} \log n \approx 0.507834 \text{.}
\end{equation*}
To see this, note that since $\vartheta$ is primitive and constant length, it follows by \Cref{THM:inf-entropy} that $h_{\operatorname{top}} (X_{\vartheta}) = h_a = h_b$, so we can calculate $h_{\operatorname{top}} (X_{\vartheta})$ by computing $h_a$ or $h_b$. 
Since for all distinct realisations $u, v \in \vartheta (a) \cup \vartheta (b)$ we have $u \neq v$ and $\vartheta$ is constant length, we have
\begin{equation*}
    \vartheta^{m+1} (a) = \vartheta^{m} (aa) \cup \vartheta^{m} (ab) \ \text{and} \ \vartheta^{m+1} (b) = \vartheta^{m} (ba)
\end{equation*}
for all $m \in \N$, where the union is disjoint. It follows that 
\begin{equation*}
    \# \vartheta^{m+1} (a) = \# \vartheta^{m} (a) (\# \vartheta^{m} (a) + \# \vartheta^{m} (b)) \ \text{and} \ \# \vartheta^{m+1} (b) = \# \vartheta^{m} (a) \# \vartheta^{m} (b) \text{,}
\end{equation*}
noting that $\vartheta^{m+1} (u) = \vartheta^{m+1} (u_1) \cdots \vartheta^{m+1} (u_{|u|})$ for all $u \in \mc{L}_{\vartheta}$. 
We next show that $\# \vartheta^{m} (a) = (m+1) \# \vartheta^{m} (b)$ for all $m \in \mathbb{N}$. 
For $m=1$, the identity clearly holds since $\# \vartheta (a) = 2$ and $\# \vartheta (b) = 1$. For $m \geq 2$, we have
\begin{equation*}
    \frac{\# \vartheta^{m} (a)}{\# \vartheta^{m} (b)} = \frac{\# \vartheta^{m-1} (a) (\# \vartheta^{m-1} (a) + \# \vartheta^{m-1} (b))}{\# \vartheta^{m-1} (a) \# \vartheta^{m-1} (b)} = \frac{\# \vartheta^{m-1} (a)}{\# \vartheta^{m-1} (b)} + 1 \text{;}
\end{equation*}
thus, it follows by induction that $\# \vartheta^{m} (a) /\# \vartheta^{m} (b) = m+1$ for all $m \in \mathbb{N}$. 
Specifically, $\# \vartheta^{m} (a) = (m+1) \# \vartheta^{m} (b)$. Hence,
\begin{equation*}
    \log ( \# \vartheta^{m} (b)) = \log (\# \vartheta^{m-1} (a) \# \vartheta^{m-1} (b)) = \log m + 2 \log (\# \vartheta^{m-1} (b)) \text{,}
\end{equation*}
and it follows inductively that
\begin{equation*}
    \frac{1}{2^m} \log ( \# \vartheta^{m} (b)) = \sum_{n=1}^{m} \frac{1}{2^n} \log n \text{.}
\end{equation*}
Letting $m \rightarrow \infty$, we obtain
\begin{equation*}
    h_{\operatorname{top}} (X_{\vartheta}) = h_b = \sum_{n=1}^{\infty} \frac{1}{2^n} \log n \text{.}
\end{equation*}
\end{example}

\subsection{Proof of Theorem \ref{THM:inf-entropy}}\label{SS:inf-entropy-proof}

If $\vartheta$ is a primitive random substitution of constant length $\ell$, then $\vartheta^m (a) \subseteq \mc L_{\vartheta}^{\ell^m} = \mc L^{\ell^m} (X_{\vartheta})$ for all $m \in \N$ and $a \in \mc A$, so $\underline{h}_a \leq \overline{h}_a \leq h_{\operatorname{top}} (X_{\vartheta})$.
Thus, to prove \Cref{THM:inf-entropy} it suffices to show that $h_{\operatorname{top}} (X_{\vartheta}) \leq \underline{h}_a$ for all $a \in \mc A$. 
The following upper bound for $h_{\operatorname{top}} (X_{\vartheta})$ is central to our proof of this inequality. 
In fact, the following does not require the assumption of primitivity and we utilise this more general formulation in the following sections when we consider non-primitive random substitutions. 
Without primitivity, it is not always the case that $\mc A \subseteq \mc L(X_{\vartheta})$.
In what follows, for each $k \in \N$ we write $\mc I_k = \cup_{a \in \mc A \cap \mc L (X_{\vartheta})} \vartheta^k (a)$.

\begin{prop}\label{PROP:entropy-upp-bd}
    Let $\vartheta$ be a random substitution of constant length $\ell$ and let $m \in \N$. 
    For each $k \in \N$, let $u^k \in \mc I_k$ be such that the quantity $\# \vartheta^m (u^k)$ is maximised.
    Then, for all $k \in \N$, the following inequality holds:
    \begin{equation*}
        h_{\operatorname{top}} (X_{\vartheta}) \leq \frac{1}{\ell^m-1} \sum_{a \in \mc A} \frac{\lvert u^k \rvert_a}{\ell^k} \log (\# \vartheta^m (a)) \text{.}
    \end{equation*}
\end{prop}
\begin{proof}
    Fix $m,k \in \N$ and let $n \in \N$. 
    By definition, every finite word in the language of the subshift is a subword of a realisation of $\vartheta^m (v)$ for some $v \in \mc L (X_{\vartheta})$.
    Moreover, since $\vartheta$ is constant length, for every legal word of length $n \ell^m$ there exists such a $v$ with length $n+1$, so
    \begin{equation}\label{EQ:legal-words-set-expanded}
        \mc L^{n \ell^m} (X_{\vartheta}) = \bigcup_{v \in \mc L^{n+1} (X_{\vartheta})} \bigcup_{j=1}^{\ell^m} \, \vartheta^m (v)_{[j,j+n \ell^m-1]} \text{.}
    \end{equation}
    Every $v \in \mc L^{n+1} (X_{\vartheta})$ can be obtained as a subword of a legal word that is the concatenation of words in $\mc I_k$. 
    Moreover, by the constant length property, there exists such a word with length $k_n = \lceil \ell^{-k} (n+1) \rceil + 2$, so
    \begin{equation*}
        \# \vartheta^m (v) \leq \prod_{a \in \mc A} \left( (\# \vartheta^m (a))^{\lvert u^k \rvert_a} \right)^{k_n} \text{.}
    \end{equation*}
    Since this bound is independent of the choice of $v \in \mc L^{n+1} (X_{\vartheta})$, it follows by \eqref{EQ:legal-words-set-expanded} that
    \begin{equation*}
        \frac{1}{n \ell^m} \log p_{X_{\vartheta}} (n \ell^m) \leq \frac{1}{n\ell^m} \log \ell^m + \frac{1}{n\ell^m} \log p_{X_{\vartheta}} (n+1) + \frac{k_n}{n \ell^m} \sum_{a \in \mc A} \lvert u^k \rvert_a \log (\# \vartheta^m (a)) \text{.} 
    \end{equation*}
    Noting that $k_n / n \rightarrow \ell^{-k}$ as $n \rightarrow \infty$, we deduce that
    \begin{equation*}
        \left( 1 - \frac{1}{\ell^m} \right) h_{\operatorname{top}} (X_{\vartheta}) \leq \frac{1}{\ell^m} \sum_{a \in \mc A} \frac{\lvert u^k \rvert_a}{\ell^k} \log (\# \vartheta^m (a)) \text{.}
    \end{equation*}
    Dividing by $1 - \ell^{-m}$ completes the proof.
\end{proof}

We now give the proof of \Cref{THM:inf-entropy}. 

\begin{proof}[Proof of Theorem \ref{THM:inf-entropy}]
    It suffices to show that $\underline{h}_a \geq h_{\operatorname{top}} (X_{\vartheta})$ for all $a \in \mc A$.
    To this end, let $n \in \N$ and, for each $k \in \N$, let $u^k \in \mc I_k$ be such that $\# \vartheta^n (u^k)$ is maximised.
    By \Cref{PROP:entropy-upp-bd}, we have
    \begin{equation*}
        h_{\operatorname{top}} (X_{\vartheta}) \leq \frac{1}{\ell^n-1} \sum_{a \in \mc A} \frac{\lvert u^k \rvert_a}{\ell^k} \log (\# \vartheta^n (a))
    \end{equation*}
    for all $k \in \N$, and so
    \begin{equation}\label{EQ:inf-thm-proof-bd}
        h_{\operatorname{top}} (X_{\vartheta}) \leq \frac{1}{\ell^n-1} \liminf_{k \rightarrow \infty} \sum_{a \in \mc A} \frac{\lvert u^k \rvert_a}{\ell^k} \log (\# \vartheta^n (a)) \text{.}
    \end{equation}
    We show that, for every $b \in \mc A$, the right hand side is bounded above by $\ell^{n} (\ell^n-1)^{-1} \underline{h}_b$.
    For each $b \in \mc A$ and $k \in \N$, let $v_b^k$ denote the realisation of $\vartheta^k (b)$ for which $\# \vartheta^n (v_b^k)$ is maximised. 
    By definition, for every $k \in \N$ there exists a $b \in \mc A$ such that $v_b^k = u^k$.
    For each $k \in \N$, let $b(k) \in \mc A$ be a letter such that $\# \vartheta^{n} (v_{b(k)}^{k}) \leq \# \vartheta^{n} (v_b^{k})$ for all $b \in \mc A$.
    By primitivity, there is an integer $K$ such that, for all $b \in \mc A$, there is a realisation $w$ of $\vartheta^K (b)$ in which every letter appears at least once.
    For each $j \in \{1,\ldots, \ell^K \}$, we have that $v_{w_j}^{k-K} \in \vartheta^{k-K} (w_j)$, so the word $v = v^{k-K}_{w_1} \cdots v^{k-K}_{w_{\ell^K}}$ is a realisation of $\vartheta^{k-K} (w)$.
    Moreover, $v$ is a realisation of $\vartheta^k (b)$.
    By the construction of $v$, we have
    \begin{equation*}
        \# \vartheta^n (v) \geq \prod_{j=1}^{\ell^K} \# \vartheta^n (v_{w_j}^{k-K}) \text{.} 
    \end{equation*}
    Since every letter in $\mc A$ appears in $w$, there is an index $i \in \{ 1, \ldots, \ell^K\}$ such that $v_{w_i}^{k-K} = u^{k-K}$.
    Noting that $\# \vartheta^n (v_{w_j}^{k-K}) \geq \# \vartheta^n (v_{b(k-K)}^{k-K})$ for all $j \neq i$, we obtain that
    \begin{equation*}
        \log (\# \vartheta^n (v)) \geq \log (\# \vartheta^n (u^{k-K})) + (\ell^K-1) \log (\# \vartheta^n (v_{b(k-K)}^{k-K})) \text{,}
    \end{equation*}
    Observe that the right hand side in the above is independent of the choice of $b \in \mc A$.
    Since, for each $b$, $v_{b}^k$ was chosen to be the realisation of $\vartheta^k (b)$ for which $\# \vartheta^n (v_{b}^k)$ is maximised, the above inequality still holds if $v$ is replaced by $v_{b}^k$.
    Thus, it follows inductively that for all $m \in \N$ and $k \geq Km$, we have
    \begin{equation}\label{EQ:inf-entropy-proof-eq}
        \log (\# \vartheta^n (v^k_{b})) \geq \sum_{j=1}^{m} (\ell^K-1)^{j-1} \log (\# \vartheta^n (u^{k-jK})) + (\ell^K - 1)^m \log (\# \vartheta^n (v_{b(k-mK)}^{k-mK})) \text{.}
    \end{equation}
    Using the above, we ascertain a lower bound on $\underline{h}_b$.
    Namely, for all $b \in \mc A$ and $m \in \N$, we have
    \begin{equation*}
        \begin{split}
            \underline{h}_b = \liminf_{k \rightarrow \infty} \frac{1}{\ell^{n+k}} \log (\# \vartheta^{n+k} (b)) &\geq \frac{1}{\ell^n} \liminf_{k \rightarrow \infty} \frac{1}{\ell^k} \log (\# \vartheta^n (v^k_b))\\
            &\geq \frac{1}{\ell^{n+K}} \sum_{j=1}^{m} \left( \frac{\ell^K-1}{\ell^K} \right)^{j-1} \liminf_{k \rightarrow \infty} \frac{1}{\ell^{k-jK}} \log (\# \vartheta^n (u^{k-jK}))\text{,}
        \end{split}
    \end{equation*}
    where in the second inequality we have used that $\vartheta^n (v_b^k) \subseteq \vartheta^{n+k} (b)$ and in the third we have applied \eqref{EQ:inf-entropy-proof-eq} and bounded the second term by zero.
    By the constant length property and \eqref{EQ:inf-thm-proof-bd},
    \begin{equation*}
        \liminf_{k \rightarrow \infty} \frac{1}{\ell^{k-jK}} \log (\# \vartheta^n (u^{k-jK})) = \liminf_{k \rightarrow \infty} \sum_{a \in \mc A} \frac{\lvert u^{k-jK}\rvert_a}{\ell^{k-jK}} \log (\# \vartheta^n (a)) \geq (\ell^n-1) h_{\operatorname{top}} (X_{\vartheta})
    \end{equation*}
    for all $j \in \{1, \ldots, m\}$. Hence,
    \begin{equation*}
            \underline{h}_b \geq \frac{\ell^n-1}{\ell^{n+K}} h_{\operatorname{top}} (X_{\vartheta}) \sum_{j=1}^{m} \left( \frac{\ell^K-1}{\ell^K} \right)^{j-1} \xrightarrow{m \rightarrow \infty} \frac{\ell^n-1}{\ell^n} h_{\operatorname{top}} (X_{\vartheta}) \text{,}
    \end{equation*}
    noting that $\sum_{j=1}^{\infty} ((\ell^K-1)/\ell^K)^{j-1} = \ell^K$.
    Since this bound holds for all $n \in \N$, we conclude that $\underline{h}_b \geq h_{\operatorname{top}} (X_{\vartheta})$, which completes the proof.
\end{proof}

\subsection{General bounds}

The coincidence of topological entropy and inflation word entropy given by \Cref{THM:inf-entropy} provides a mechanism for calculating the topological entropy for a broad class of non-compatible random substitution subshifts.
However, a limitation of \Cref{THM:inf-entropy} is that it does not give any information about the rate of convergence.
Moreover, it does not provide a means of calculating topological entropy for non-primitive random substitutions.
In this section, we prove general bounds on the topological entropy for constant length random substitutions, which provide a means of obtaining good estimates, even in cases where a closed form cannot be obtained via \Cref{THM:inf-entropy}.
Central to our approach is the following definition.

\begin{definition}
    Let $\vartheta$ be a random substitution over some alphabet $\mc A$. We say that a vector $\mathbf{\nu} = (\nu_a)_{a \in \mc A} \in [0,1]^{\# \mc A}$ is a \emph{permissible letter frequency vector} for $\vartheta$ if there exists a letter $b \in \mc A \cap \mc L (X_{\vartheta})$, a sequence of integers $(n_k)_k$ and $v^k \in \vartheta^{n_k} (b)$ such that $\lvert v^k \rvert_a / \lvert v^k \rvert \rightarrow \nu_a$ for all $a \in \mc A$.
\end{definition}

The assumption that $b \in \mc L (X_{\vartheta})$ guarantees that the sequence $v^k$ is in the language of the subshift $X_{\vartheta}$. 
The inclusion $\mc A \subseteq \mc L (X_{\vartheta})$ always holds under primitivity. 
However, there are non-primitive random substitutions for which there exists a letter that is not in the language of the subshift. 
For example, for the random substitution $\vartheta \colon a \mapsto \{ab\},\, b \mapsto \{bc\},\, c \mapsto \{cb,cc\}$ we have that $a \notin \mc L (X_{\vartheta})$, since no letter can legally precede it.

Recall from \Cref{LEM:PF-bounds} that, for compatible primitive random substitutions, there exists a unique asymptotic growth rate $\lambda > 1$ and permissible letter frequency vector $\mathbf{R} = (R_a)_{a \in \mc A}$.
In this setting, Gohlke \cite{gohlke} showed that the following bounds hold for all $m \in \N$:
\begin{equation*}
    \frac{1}{\lambda^m} \sum_{a \in \mc A} R_a \log(\# \vartheta^m (a)) \leq h_{\operatorname{top}} (X_{\vartheta}) \leq \frac{1}{\lambda^m-1} \sum_{a \in \mc A} R_a \log(\# \vartheta^m (a)) \text{.}
\end{equation*}
Under compatibility, the only permissible letter frequency vector is $\mathbf{R}$.
However, without compatibility there often exists a continuum of permissible letter frequency vectors.
We prove an analogue of the above bounds for constant length random substitutions that holds without compatibility, using the notion of a permissible letter frequency vector.
In the lower bound, we can replace $\mathbf{R}$ with any permissible letter frequency vector; however, for the upper bound we require a particular choice.

\begin{prop}\label{PROP:plfv-entropy-bounds}
    Let $\vartheta$ be a random substitution of constant length $\ell \geq 2$.
    Then, for every permissible letter frequency vector $\nu = (\nu_a)_{a \in \mc A}$,
    \begin{equation*}
        h_{\operatorname{top}} (X_{\vartheta}) \geq \frac{1}{\ell} \sum_{a \in \mc A} \nu_a \log (\# \vartheta (a)) \text{.}
    \end{equation*}
    Moreover, there exists a permissible letter frequency vector $\eta = (\eta_a)_{a \in \mc A}$ such that
    \begin{equation*}
        h_{\operatorname{top}} (X_{\vartheta}) \leq \frac{1}{\ell-1} \sum_{a \in \mc A} \eta_a \log (\# \vartheta (a)) \text{.}
    \end{equation*}
    In particular, $\eta$ can be taken to be any permissible letter frequency vector that maximises the quantity $\sum_{a \in \mc A} \eta_a \log (\# \vartheta (a))$.
\end{prop}

\begin{proof}
    We first prove the lower bound.
    For any permissible letter frequency vector $\nu$, there exists a letter $b \in \mc A \cap \mc L (X_{\vartheta})$, a sequence of positive integers $(n_k)_k$ with $n_k \rightarrow \infty$ and realisations $v^k \in \vartheta^{n_k} (b)$ such that $\lvert v^k \rvert_a / \lvert v^k \rvert \rightarrow \nu_a$ as $k \rightarrow \infty$, for all $a \in \mc A$. 
    Since $b \in \mc A$ and $\vartheta (X_{\vartheta}) \subseteq X_{\vartheta}$, every realisation of $\vartheta (v^k)$ is in the language of the subshift $X_{\vartheta}$.
    As every realisation of $\vartheta (v^k)$ has length $\ell^{n_k+1}$, it follows that $\# \mc L^{\ell^{n_k+1}} (X_{\vartheta}) \geq \# \vartheta (v^k)$. 
    Moreover, the constant length property gives that $\# \vartheta (v^k) = \prod_{a \in \mc A} (\# \vartheta (a))^{\lvert v^k \rvert_a}$, so we have
    \begin{equation*}
        \frac{1}{\ell^{n_k+1}} \log (\# \mc L^{\ell^{n_k+1}} (X_{\vartheta})) \geq \frac{1}{\ell} \sum_{a \in \mc A} \frac{\lvert v^k \rvert_a}{\ell^{n_k}} \log (\# \vartheta (a)) \text{.}
    \end{equation*}
    Letting $k \rightarrow \infty$, we obtain
    \begin{equation*}
        h_{\operatorname{top}} (X_{\vartheta}) \geq \frac{1}{\ell} \sum_{a \in \mc A} \nu_a \log (\# \vartheta (a)) \text{.}
    \end{equation*}
    The upper bound is largely a consequence of \Cref{PROP:entropy-upp-bd}.
    Letting $(u^k)_k$ denote a sequence of exact inflation words, with $u^k \in \mc I_k$ for all $k \in \N$,
    such that $\# \vartheta (u^k)$ is maximised, \Cref{PROP:entropy-upp-bd} gives
    \begin{equation}\label{EQ:plfv-proof-entropy-bd}
        h_{\operatorname{top}} (X_{\vartheta}) \leq \frac{1}{\ell-1} \sum_{a \in \mc A} \frac{\lvert u^k \rvert_a}{\ell^k} \log (\# \vartheta^m (a)) \text{.}
    \end{equation}
    By the pigeonhole principle, there exists a letter $b \in \mc A \cap \mc L(X_{\vartheta})$ such that $u^k \in \vartheta^k (b)$ for infinitely many $k$.
    Thus, by the compactness of $[0,1]^{\# \mc A}$, there exists a sequence of positive integers $(k_n)_n$ such that $u^{k_n} \in \vartheta^{k_n} (b)$ for all $n \in \N$ and $\lvert u^{k_n} \rvert_a / \ell^{k_n}$ converges for all $a \in \mc A$. 
    By definition, this limit is a permissible letter frequency vector $\eta$.
    Passing to limits along the subsequence $(k_n)_n$ in \eqref{EQ:plfv-proof-entropy-bd}, we obtain that
    \begin{equation*}
        h_{\operatorname{top}} (X_{\vartheta}) \leq \frac{1}{\ell-1} \sum_{a \in \mc A} \eta_a \log (\# \vartheta (a)) \text{.}
    \end{equation*}
    In particular, the above holds for any permissible letter frequency vector that maximises the quantity $\sum_{a \in \mc A} \eta_a \log (\# \vartheta (a))$. 
    This completes the proof.
\end{proof}

As a consequence of \Cref{PROP:plfv-entropy-bounds}, we obtain the following characterisation of positivity of topological entropy for constant length random substitutions. 
We emphasise that, in contrast to the results in \Cref{S:positivity}, we do not assume primitivity in the following.

\begin{corollary}\label{COR:CL-TE-positivity}
    If $\vartheta$ is a constant length random substitution, then $h_{\operatorname{top}} (X_{\vartheta}) > 0$ if and only if there exists a permissible letter frequency vector $\nu$ and $a \in \mc A$ such that $\nu_a > 0$ and $\# \vartheta (a) \geq 2$.
\end{corollary}

We conclude this section by presenting an example of a primitive random substitution where it is not clear how to obtain an exact formula for the topological entropy from \Cref{THM:inf-entropy}, but where the bounds given by \Cref{PROP:plfv-entropy-bounds} allow a good estimate to be obtained.

\begin{example}\label{EX:approx-entropy}
    Let $\vartheta$ be the primitive random substitution defined by $\vartheta \colon a \mapsto \{aa,bb\}, \, b \mapsto \{ab\}$
    and let $X_{\vartheta}$ denote the associated subshift. 
    Since for all $m \in \N$, there exists a realisation of $\vartheta^m (a)$ in which $a$ is the only letter that appears, the vector $(1,0)^{T}$ is a permissible letter frequency vector for every power of $\vartheta$. 
    For all $m \in \N$, $\# \vartheta^m (a) \geq \# \vartheta^m (b)$, so for every permissible letter frequency vector $\nu = (\nu_a,\nu_b)$, we have $\nu_a \log (\# \vartheta^m (a)) + \nu_b \log (\# \vartheta^m (b)) \leq \log (\# \vartheta^m (a))$.
    Hence, it follows by \Cref{PROP:plfv-entropy-bounds} (applied to $\vartheta^m$) that
    \begin{equation*}
        \frac{1}{2^m} \log (\# \vartheta^m (a)) \leq h_{\operatorname{top}} (X_{\vartheta}) \leq \frac{1}{2^m -1} \log (\# \vartheta^m (a)) \text{.}
    \end{equation*}
    For all $m \in \N$, the inflation sets $\vartheta^m (a)$ satisfy the following relations:
    \begin{equation*}
        \vartheta^m (a) = \vartheta^{m-1} (aa) \cup \vartheta^{m-1} (bb) \quad \text{and} \quad \vartheta^m (b) = \vartheta^{m-1} (ab) \text{.}
    \end{equation*}
    Since $\vartheta$ is constant length and $\vartheta (a) \cap \vartheta (b) = \varnothing$, we have that $\vartheta^m (a) \cap \vartheta^m (b) = \varnothing$ for all $m \in \N$.
    Hence, it follows from the above relations that
    \begin{equation*}
        \# \vartheta^m (a) = \left( \# \vartheta^{m-1} (a) \right)^2 + \left( \# \vartheta^{m-1} (b) \right)^2 = \left( \# \vartheta^{m-1} (a) \right)^2 + \prod_{i=1}^{m-2} \left( \# \vartheta^i (a) \right)^2 \text{.}
    \end{equation*}
    A computer-assisted calculation for the case $m = 14$ gives that, to four decimal places, the topological entropy of the subshift $X_{\vartheta}$ is $h_{\operatorname{top}} (X_{\vartheta}) = 0.4115$.
\end{example}

\section{Complexity of constant length random substitution subshifts}\label{S:CL-complexity}

Classifying the functions that can be obtained as the complexity function of a subshift is a central problem in symbolic dynamics \cite{combinatorics-automata-book,ferenczi}. 
One of the most famous results in this direction is a consequence of the classical Morse--Hedlund theorem \cite{morse-hedlund}, and states that the complexity function of any subshift is either bounded above by a constant or grows at least linearly.
For subshifts of deterministic substitutions, it was shown by Pansiot \cite{pansiot} that the complexity function is either $\Theta(1)$, $\Theta(n)$, $\Theta(n \log \log n)$, $\Theta(n \log n)$ or $\Theta(n^2)$.
Moreover, under primitivity it is always $\Theta(1)$ or $\Theta(n)$.
Similarly, for primitive random substitutions, \Cref{COR:no-int-growth} gives that the complexity function either grows exponentially or is $\Theta(1)$ or $\Theta(n)$.

Without primitivity, the picture is very different.  
In \Cref{PROP:CL-poly-growth-powers}, we show that the set of $\alpha$ for which there exists a constant length random substitution with $\Theta(n^{\alpha})$ complexity function is dense in $[1,\infty)$.
Thus, in stark contrast to the deterministic and primitive random settings, there is no polynomial complexity gap.
Further, we show that constant length random substitutions can give rise to subshifts with intermediate growth complexity function, which primitivity forbids by \Cref{COR:no-int-growth}.
We provide sufficient conditions for intermediate growth in \Cref{PROP:int-growth-suff-cond}.

\begin{prop}\label{PROP:CL-poly-growth-powers}
    The set of $\alpha$ for which there exists a constant length random substitution subshift with $\Theta(n^{\alpha})$ complexity function is dense in $[1,\infty)$.
\end{prop}
\begin{proof}
    We show that there exists a set $A$, dense in $[1,\infty)$, such that for every $\alpha \in A$ there exists a constant length random substitution whose subshift has $\Theta(n^{\alpha})$ complexity function.
    To this end, let $\ell \geq 3$, let $\mc A = \{a_1,\ldots,a_{\ell+2}\}$ be an alphabet of $\ell+2$ letters, and let
    \begin{equation*}
        \mc R_{\ell} = \left\{ u \in \mc A^{\ell} \colon \text{ for all $i \in \{1,\ldots,\ell\}$, $a_i$ appears in $u$ precisely once} \right\}
    \end{equation*}
    be the set of all permutations of $\{a_1,\ldots,a_{\ell}\}$.
    Note that every $u \in \mc R_{\ell}$ has length $\lvert u \rvert = \ell$.
    Given a non-empty subset $\mc S$ of $\mc R_{\ell}$, let $\vartheta_{\ell,\mc S}$ be the random substitution  of constant length $\ell$ defined over the alphabet $\mc A$ by
    \begin{equation*}
        \vartheta_{\ell, \mc S} \colon
        \begin{cases}
            a_i \mapsto a_i \cdots a_i \quad \text{for all $i \in \{1,\ldots,\ell\}$,}\\
            a_{\ell+1} \mapsto \mc S \text{,}\\
            a_{\ell+2} \mapsto a_{\ell+1} a_{\ell+2} a_1 \cdots a_1 \text{,}
        \end{cases}
    \end{equation*}
    and let $X_{\ell,\mc S}$ denote the corresponding subshift.
    For notational convenience, we write $p_{\ell, \mc S}$ for the complexity function of $X_{\ell, \mc S}$.
    We show that $p_{\ell, \mc S}$ is $\Theta(n^{1+\log_{\ell} \mc (\# S)})$.
    We first prove this in the case $n = \ell^k$ for some $k \in \N$ and then extend to all $n \in \N$ in the second step.
    To this end, we first compute the cardinalities $\# \vartheta^m (a_i)$ for each $m \in \N$ and $i \in \{1,\ldots,\ell+2\}$. 
    Observe that $\# \vartheta^m (a_i) = 1$ for all $m \in \N$ and $i \in \{1,\ldots,\ell\}$.
    Further, since every word in the set $\mc S$ contains only letters from the set $\{a_1,\ldots,a_{\ell}\}$, we have that $\# \vartheta^m (a_{\ell+1}) = \# \mc S$ for all $m \in \N$.
    Finally, we have
    \begin{equation*}
        \# \vartheta^m (a_{\ell+2}) = (\# \vartheta^{m-1} (a_{\ell+1})) (\# \vartheta^{m-1} (a_{\ell+2})) = (\# \mc S) (\# \vartheta^{m-1} (a_{\ell+2})) = (\# \mc S)^{m-1}
    \end{equation*}
    for all $m \in \N$.
    Every letter in the alphabet $\mc A$ is in the language of the subshift, so $\mc L (X_{\ell,\mc S}) = \mc L_{\vartheta_{\ell,\mc S}}$ and thus $p_{\ell, \mc S} (n) = \# \mc L_{\vartheta}^n$ for all $n \in \N$.
    By the constant length property, if $k \in \N$, $m \in \{1,\ldots,k\}$ and $u$ is a legal word of length $\ell^k$, then there exists a legal word $v$ of length $\ell^{k-m}+1$ and an integer $j \in \{1,\ldots,\ell^{k-m}\}$ such that $u \in \vartheta^m (v)_{[j,j+\ell^k]}$;
    hence,
    \begin{equation}\label{EQ:legal-words-decomp}
        \mc L_{\vartheta_{\ell,\mc S}}^{\ell^k} = \bigcup_{v \in \mc L_{\vartheta_{\ell,\mc S}}^{\ell^{k-m}+1}} \bigcup_{j=1}^{\ell^m} \vartheta^m (v)_{[j,j+\ell^k-1]} \text{.}
    \end{equation}
    For all $k \geq 2$, we have that $\# \vartheta^k (a_{\ell+2}) \geq \# \vartheta^k (a_{\ell+1}) \geq \# \vartheta^k (a_i)$ for all $i \in \{1,\ldots,\ell\}$.
    Since $a_{\ell+2} a_{\ell+2}$ is not a legal word, we have that $\# \vartheta^k (v) \leq \# \vartheta^k (a_{\ell+1}) \# \vartheta^{k} (a_{\ell+2}) = (\# \mc S)^k$ for all $v \in \mc L_{\vartheta}^2$.
    Hence, it follows by \eqref{EQ:legal-words-decomp} in the case $m=k$ that
    \begin{equation}\label{EQ:legal-words-UB}
        p_{\ell, \mc S} (\ell^k) \leq p_{\ell, \mc S} (2) \ell^k (\# \mc S)^k = p_{\ell, \mc S} (2) \, (\ell^k)^{1 + \log_{\ell} (\# S)}  \text{.}
    \end{equation}
    For the lower bound, observe that the $\ell+1$ letter word $w = a_{\ell+1} a_{\ell+2} a_1 \cdots a_1$ is legal, so it follows by \eqref{EQ:legal-words-decomp} in the case $m=k-1$ that
    \begin{equation}\label{EQ:legal-words-supset}
        \mc L_{\vartheta_{\ell,\mc S}}^{\ell^k} \supseteq \bigcup_{j=1}^{\ell^{k-1}} \vartheta^{k-1} (w)_{[j,j+\ell^k-1]} \text{.}
    \end{equation}
    Every realisation of $\vartheta^{k-1} (w)$ contains precisely one occurrence of the letter $a_{\ell+2}$. 
    Further, there is a positive integer $\ell^{k-1} < n < 2 \ell^{k-1}$ such that for every realisation of $\vartheta^{k-1} (w)$, the letter $a_{\ell+2}$ appears in position $n$.
    Thus, the above union is disjoint.
    Moreover, for all $j \in \{1,\ldots,\ell^{k-1}\}$, the inflated image of the second letter of $w$ is contained in the corresponding realisation of $\vartheta^{k-1} (w)_{[j,j+\ell^k-1]}$, so $\# \vartheta^{k-1} (w)_{[j,j+\ell^k-1]} \geq \# \vartheta^{k-1} (a_{\ell+2}) = (\# \mc S)^{k-2}$.
    Hence, it follows by \eqref{EQ:legal-words-supset} that
    \begin{equation}\label{EQ:legal-words-LB}
        p_{\ell,\mc S} (\ell^k) \geq \ell^{k-1} (\# \mc S)^{k-2} = \ell^{-1} (\# \mc S)^{-2} (\ell^k)^{1+\log_{\ell} (\# \mc S)} \text{.}
    \end{equation}
    Now, let $n \in \N$ and let $k$ be the unique integer such that $\ell^k \leq n < \ell^{k+1}$. 
    By the monotonicity of the complexity function, we have
    \begin{equation*}
        p_{\ell, \mc S} (\ell^k) \leq p_{\ell, \mc S} (n) \leq p_{\ell, \mc S} (\ell^{k+1})
    \end{equation*}
    and so it follows by \eqref{EQ:legal-words-UB} and \eqref{EQ:legal-words-LB} that 
    \begin{equation*}
        \ell^{-2} (\# \mc S)^{-3}
        n^{1+\log_{\ell} (\# \mc S)} \leq p_{\ell, \mc S} (n) \leq \ell (\# \mc S) \, p_{\ell, \mc S} (2) \, n^{1+\log_{\ell} (\# \mc S)} \text{.}
    \end{equation*}
    Hence, we conclude that the complexity function of $X_{\ell,\mc S}$ is $\Theta(n^{1+\log_{\ell} (\# \mc S)})$.
    Since the set
    \begin{equation*}
        A = \left\{ 1 + \log_{\ell} k \colon \ell \in \{3,4,5,\ldots\}, \, k \in \{1,\ldots,\ell !\} \right\}
    \end{equation*}
    is dense in $[1,\infty)$ and, for every $\alpha \in A$, there exists an $\ell \in \{3,4,5,\ldots\}$ and a subset $\mc S$ of $\mc R_{\ell}$ such that the subshift $X_{\ell, \mc S}$ has $\Theta(n^{\alpha})$ complexity function, the result follows.
\end{proof}

\begin{example}
    Let $\vartheta$ be the random substitution defined by $\vartheta \colon a \mapsto \{aaa\}, \, b \mapsto \{bbb\}, \, c \mapsto \{ccc\}, \, d \mapsto \{abc,\, acb\}, \, e \mapsto \{dea\}$. 
    The associated subshift has $\Theta(n^{1+\log_3 2})$ complexity function.
\end{example}

In the following, we provide sufficient conditions for a constant length random substitution to give rise to a subshift with intermediate growth complexity function.

\begin{prop}\label{PROP:int-growth-suff-cond}
    Let $\vartheta$ be a constant length random substitution for which the following hold:
    \begin{itemize}
        \item for all $a \in \mc A$ with $\# \vartheta (a) \geq 2$ and all permissible letter frequency vectors $\nu$, we have $\nu_a = 0$;
        \item there exists a letter $b \in \mc A$ for which $\# \vartheta (b) \geq 2$ and a realisation $v \in \vartheta (b)$ with $\lvert v \rvert_b \geq 2$.
    \end{itemize}
    Then, the associated subshift $X_{\vartheta}$ has intermediate growth complexity function.
\end{prop}
\begin{proof}
    The first condition guarantees that $h_{\operatorname{top}} (X_{\vartheta}) = 0$ by \Cref{COR:CL-TE-positivity}.
    Meanwhile, it follows inductively from the second condition that for all $m \in \N$, there exists a realisation $v^m \in \vartheta^m (b)$ with $\lvert v^m \rvert_b \geq 2^{m}$.
    This gives that the letter $b$ is in the language of the subshift, so for all $m \in \N$ we have $\vartheta^m (b) \subseteq \mc L^{\ell^m} (X_{\vartheta})$; hence,
    \begin{equation*}
        p_{X_{\vartheta}} (\ell^m) \geq \# \vartheta^m (b) \geq \# \vartheta (v^{m-1}) \geq 2^{\lvert v^m \rvert_b} \geq 2^{2^{m}} = 2^{ (\ell^m)^{\log_{\ell} 2}} \text{.}
    \end{equation*}
    Thus, if $n \in \N$ and $m$ is the integer such that $\ell^m \leq n < \ell^{m+1}$, then by monotonicity of the complexity function we have
    \begin{equation*}
        p_{X_{\vartheta}} (n) \geq p_{X_{\vartheta}} (\ell^m) \geq 2^{(\ell^m)^{\log_{\ell} 2}} \geq 2^{2^{-1} n^{\log_{\ell} 2}} \text{,}
    \end{equation*}
    which grows faster than any polynomial.
    Since $h_{\operatorname{top}} (X_{\vartheta}) = 0$, we also have that $p_{X_{\vartheta}}$ grows sub-exponentially.
    Hence, we conclude that $p_{X_{\vartheta}}$ has intermediate growth.
\end{proof}

\begin{example}\label{EX:int-growth}
Let $\vartheta$ be the random substitution defined by $\vartheta \colon a \mapsto \{aaa\},\, b \mapsto \{abb,bba\}$ and let $X_{\vartheta}$ denote the corresponding subshift. 
We have that $\lvert \vartheta^m (a) \rvert_b = 0$ and $\lvert \vartheta^m (b) \rvert_b = 2^m$ for all $m \in \N$, so for every realisation $v \in \vartheta^m (a) \cup \vartheta^m (b)$ we have $\lvert v \rvert_b / 3^m \leq 2^m / 3^m \rightarrow 0$. 
Hence, the only permissible letter frequency vector is $(1,0)^{T}$; since $\# \vartheta (a) = 1$, the first condition of \Cref{PROP:int-growth-suff-cond} is satisfied. 
Moreover, since $abb \in \vartheta (b)$ and $\# \vartheta (b) \geq 2$, the second condition is also satisfied.
Therefore, it follows by \Cref{PROP:int-growth-suff-cond} that $X_{\vartheta}$ has intermediate growth complexity function.
In fact, by following similar arguments to those used in the proof of \Cref{PROP:CL-poly-growth-powers}, it can be shown that there exist constants $c_1, c_2 > 0$ such that, for all $n \in \N$,
\begin{align*}
    c_1 2^{n^{\log_{3} 2}} \leq p_{X_{\vartheta}} (n) \leq c_2 n 2^{n^{\log_{3} 2}} \text{.}
\end{align*} 
\end{example}

\section*{Acknowledgments}

The author thanks Philipp Gohlke, Dan Rust and Tony Samuel for valuable comments on a draft version of this paper.
He is also grateful for financial support from EPSRC DTP, the University of Birmingham, and EPSRC grant EP/Y023358/1, which supported this work.

\bibliographystyle{abbrv}
\bibliography{ref}

\end{document}